\documentclass[a4paper, 12pt]{amsart}
\usepackage{graphicx}
\usepackage[colorlinks, linkcolor= blue, citecolor= red]{hyperref}
\usepackage{pdfsync}
\usepackage{color}

\usepackage[T1]{fontenc}
\usepackage[utf8]{inputenc} 

\usepackage[all]{xy}
\usepackage{amsmath, amssymb, amsfonts, amscd, amsthm, mathrsfs}

\renewcommand{\phi}{\varphi}
\newcommand{\C}{\mathbb{C}}
\newcommand{\z}{\mathbb{Z}}

\newcommand{\q}{\mathbb{Q}}
\newcommand{\f}{\mathbb{F}}

\newcommand{\mo}[1]{\mod{#1}}

\newtheoremstyle{pedro}{}{}{\itshape}{}{\sc}{~--}{ }{\thmname{#1}\thmnumber{ #2}\thmnote{ (#3)}}

\newtheoremstyle{pedrodef}{}{}{}{}{\sc}{~--}{ }{\thmname{#1}\thmnumber{ #2}\thmnote{ (#3)}}

\theoremstyle{pedro}
\newtheorem{lem}{Lemma}[section]

\newtheorem{thm}[lem]{Theorem}

\newtheorem{prop}[lem]{Proposition}

\newtheorem{coro}[lem]{Corollary}

\theoremstyle{remark}

\newtheorem{rmk}[lem]{Remark}

\theoremstyle{pedrodef}

\title[Stirling numbers]{Not so new congruences for Stirling numbers of the first kind, with an application to Chern classes}

\author{Pierre Guillot}
\author{Yohann Ségalat}

\address{
Universit\'{e} de Strasbourg \& CNRS\\
Institut de Recherche Math\'{e}matique Avanc\'{e}e\\
7~Rue Ren\'{e} Descartes\\
67084 Strasbourg, France}

\email{guillot@math.unistra.fr segalat@math.unistra.fr}

\let\oldtocsection=\tocsection
\let\oldtocsubsection=\tocsubsection
\let\oldtocsubsubsection=\tocsubsubsection

\renewcommand{\tocsection}[2]{\hspace{0em}\oldtocsection{#1}{#2}}
\renewcommand{\tocsubsection}[2]{\hspace{2em}\oldtocsubsection{#1}{#2}}
\renewcommand{\tocsubsubsection}[2]{\hspace{2em}\oldtocsubsubsection{#1}{#2}}

\numberwithin{equation}{section}

\begin{document}

\maketitle

%% %\setcounter{tocdepth}{1}
%% %\tableofcontents

\begin{abstract}
In this paper we give simple expressions, involving binomial coefficients, for the value of~$c(p^r m, k)$ modulo~$p^r$, where~$p$ is a prime not dividing~$m$ and~$r \ge 1$. Here we write~$c(n,k)$ for Stirling numbers of the first kind.

As an application, we compute the Chern classes of permutation representations of cyclic groups.
\end{abstract}

\noindent{\bfseries Important information.} After we had submitted a first version of this paper for publication (with the title ``New congruences...''), we learned from the referee about the paper by C\'ardenas and Lluis, see~\cite{cardenas}. This work from 1987 contains the bulk of our results. More precisely, our Theorem~\ref{thm-poly-odd-prime} is proved, and in the language of Chern classes ; the rest can be deduced quite easily, and no good reasons seem to remain to publish the present paper.

Situations like this are pretty rare nowadays. Consider that:

\begin{itemize}
\item a search on MathSciNet about Stirling numbers, or Chern classes of cyclic groups, would not (and did not) bring up the paper by C\'ardenas and Lluis ;
\item that paper's title, or its review on MathSciNet, do not indicate that it should contain our results ;
\item a question asked on MathOverflow about Stirling numbers (which the reader can find easily) led us to believe quite firmly that our congruences had never appeared -- experts are usually so adept at giving references on MathOverflow !
\item a private communication with the author of the 2015 paper~\cite{recent} on Stirling numbers confirmed our belief in the novelty of our work (whereas in fact, Theorem 1.3 in this paper is weaker than the C\'ardenas and Lluis result).
\end{itemize}

It will be useful, at least, to keep this paper in the Arxiv. This way, people searching for the congruences of the kind we present will have a hit, and will be told about the C\'ardenas and Lluis paper.

\section{Introduction}

We shall write~$c(n,k)$ for the signless Stirling numbers of the first kind, for which there are so many possible definitions (see the discussion in \S1.3 of~\cite{stan}). While many congruences are available in the literature involving Stirling numbers {\em of the second kind}, much less seems to be known about~$c(n,k)$ in this respect. See~\cite{howard}, \cite{howard2}, \cite{peele}, and more recently~\cite{recent}, whose results overlap partly with ours. 

Let us write~$v_p(n)$ for the highest power of~$p$ dividing~$n$, when~$n \ne 0$. In this paper we are interested in~$c(n, k)$ modulo~$p^r$ with $r= v_p(n)$, when~$r > 0$ -- that is, when~$p$ divides~$n$. We obtain explicit, simple expressions involving binomial coefficients.

\begin{thm} \label{thm-main}
Let~$n$ be an integer, let~$p$ be an odd prime dividing~$n$, and let~$r= v_p(n)$. If~$k$ is an integer such that~$n-k$ is not divisible by~$p-1$, then one has 
\[ c(n, k) \equiv 0 \mo{p^r} \, .   \]
If on the other hand~$n-k = \ell(p-1)$ then one has 
\[ c(n, k) \equiv (-1)^\ell { n/p \choose \ell} \mo{p^r} \, .   \]

Now assume that~$n$ is even and let~$r= v_2(n)$. Suppose first that $r \ge 3$.
If~$n - k = 2\ell + 1$, then 
\[ c(n, k) \equiv (-1)^\ell {n/4 \choose \ell} \frac{n} {2} \mo{2^r} \, .   \]
If~$n-k = 2\ell$, then 
\[ c(n,k) \equiv (-1)^\ell \left[  {n/4 \choose \ell} - \frac{n} {2} {n/4 \choose \ell -1 }    \right]   \mo{2^r} \, .   \]
For small values of~$r$ these expressions must be modified as follows. For~$r=1$ one has
\[ c(n, k) \equiv { n/2 \choose n-k} \mo{2} \, .   \]
For~$r=2$, if~$n-k = 2 \ell$ one has 
\[ c(n,k) \equiv (-1)^\ell {n/4 \choose \ell} \mo{4} \, ,   \]
while if~$n-k = 2\ell + 1$ one has 
\[ c(n,k) \equiv (-1)^\ell {n/4 \choose \ell} \frac{n} {2} \mo{4} \, .   \]
\end{thm}

Note that this Theorem improves on Theorem 1.3 in~\cite{recent}, which gives only congruences mod~$p^e$ with~$e \le \frac{r+1} {2}$, rather than mod~$p^r$. 

Since all the primes dividing~$n$ have been considered, the Chinese remainder theorem implies that, at least in principle, the number~$c(n,k)$ is entirely determined modulo~$n$ by our formulas.

 Much is known about~$v_p\left[{m \choose k}\right]$: for example it is classical that it equals the number of carries when you compute the sum of~$k$ and~$m-k$ in base~$p$ (this is known as Kummer's theorem). A particular case is that~$v_p\left[{p^r \choose k}\right] = r - v_p(k)$ for~$k > 0$. The following corollaries are then simple exercises, and they seem worth stating separately.

\begin{coro} \label{coro-vp-stirling-odd}
Let~$p$ be an odd prime and~$r>0$. If~$k$ is an integer such that~$p^r -k$ is not divisible by~$p-1$, then~$v_p(c(p^r, k)) \ge r$.

If on the other hand~$p^r - k = \ell (p-1)$, then~$v_p(c(p^r, k)) = r-1 - v_p(\ell)$.

\end{coro}

For~$p=2$:

\begin{coro}
Let~$r \ge 3$. If~$2^r - k = 2\ell +1$ then~$v_2(c(2^r, k)) \ge r$ unless~$\ell= 0$ or~$2^{r-2}$, and in either case~$v_2(c(2^r, k)) = r-1$.

If on the other hand~$2^r - k = 2\ell$, then~$v_2(c(2^r, k)) = r - 2 - v_2(\ell)$.

%% If~$2^r - k = 2\ell$, then~$v_2(c(2^r, k))= r - 2 - v_2(\ell)$. 

%% If~$2^r - k$ is odd, we distinguish two cases. If~$n-k \not\in \{ 1, 2, 2^{r-1} + 1, 2^{r-1} +2 \}$, then~$v_2(c(2^r, k)) \ge r$. Otherwise, if $n-k \in \{ 1, 2, 2^{r-1} + 1, 2^{r-1} +2 \}$, then $v_2(c(2^r, k))= r-1$.
\end{coro}

%% In particular if~$p$ is any prime, we see that~$c(p^r, k)$ is always divisible by~$p$ unless~$k$ is either~$p^r$ or~$p^{r-1}(p-1)$. However this particular case is easy to prove directly.

As an application, we propose the computation of some Chern classes. The terminology in the following theorem will be explained in the text.

\begin{thm} \label{thm-chern-intro}
Let~$G$ be a cyclic group of order~$p^r$, where~$p$ is an odd prime. Let~$X$ be a transitive~$G$-set, and let~$V=\C[X]$ be the corresponding permutation representation of~$G$.

If~$X$ is not free, that is if~$V$ is not the regular representation of~$G$, then the Chern classes of~$V$ all vanish.

If~$V$ is the regular representation, one has~$c_k(V) \ne 0$ if and only if~$p-1$ divides~$k$.
\end{thm}

As we shall explain, what one needs to know in order to compute the Chern classes of {\em any} permutation representation of a cyclic group is precisely~$c(n,k)$ modulo~$n$, for various values of~$n$ and~$k$. Thus our results in this paper entirely answer the question, and the above theorem is just an example of an easy statement.

The paper is organized as follows. In section~\ref{sec-poly}, we state a few polynomial identities, which imply easily the congruences of theorem~\ref{thm-main}. What is more, these identities may be easier to commit to memory. Then in section~\ref{sec-p-odd}, we prove the identities when~$p$ is odd, except for a little result of commutative algebra (to put things emphatically) which is established in section~\ref{sec-cyclo}. It is in section~\ref{sec-p2} that we turn to the case~$p=2$, which is relatively easy once the statement is known. Finally in section~\ref{sec-chern} we present the applications to Chern classes.      

\section{Polynomial identities} \label{sec-poly}

Throughout the paper we shall use the polynomials 
\[ P_n(t) = \prod_{k=0}^{n-1} \, (1 + kt) \in \z[t] \, .   \]
We shall take as our definition of the numbers~$c(n,k)$ the identity 
\[ P_n(t) = \sum_{k=0}^n c(n,n-k) t^k  = \sum_{k=0}^n c(n, k) t^{n-k}\, .   \]
See Proposition 1.3.7 in~\cite{stan} (and the rest of \S1.3 there) for a proof that this agrees with other classical definitions.

One of our main results is the following.

\begin{thm} \label{thm-poly-odd-prime}
Let~$p$ be an odd prime and~$r > 0$. Then 
\[ P_{p^r}(t) \equiv (1 - t^{p-1})^{p^{r-1}} \mo{p^r} \, .   \]
\end{thm} 

(Here unsurprisingly we write~$A(t) \equiv B(t) \mo{N}$ to indicate that the difference~$A(t) - B(t)$ is a polynomial whose coefficients are all divisible by~$N$.)

We will prove this in the next section. However we can at once deal with some consequences:

\begin{coro}
Let~$p$ be an odd prime and suppose that~$r= v_p(n) > 0$. Then 
\[ P_n(t) \equiv (1 - t^{p-1})^{\frac{n} {p}} \mo{p^r}\, .  \]
As a result, the congruences announced in theorem~\ref{thm-main} hold. 
%% Further, if~$n-k$ is not divisible by~$p-1$ one has 
%% %
%% \[ c(n, k) \equiv 0 \mo{p^r} \, .   \]
%% %
%% If on the other hand~$n-k = \ell(p-1)$ then one has 
%% %
%% \[ c(n, k) \equiv (-1)^\ell { \frac{n} {p} \choose \ell} \mo{p^r} \, .   \]
%% %
\end{coro}

\begin{proof}[Proof of the corollary]
Clearly the congruence of polynomials implies the congruences between the coefficients. From theorem~\ref{thm-poly-odd-prime} we have 
\[ P_{p^r}(t) \equiv (1 - t^{p-1})^{p^{r-1}} \mo{p^r} \, .   \]
The corollary then follows from the simple remark that 
\[ P_n(t) \equiv P_{p^r}(t)^m \mo{p^r} \, ,   \]
where we have written~$n= p^r m$.
\end{proof}

When~$p=2$, the statement is a little different.

\begin{thm} \label{thm-poly-p2}
Let~$r \ge 3$. Then 
\begin{align*}
P_{2^r} (t) & \equiv (1-t^2)^{2^{r-2}} + 2^{r-1} (t+t^2+t^{2^{r-1} +1 } + t^{2^{r-1} +2 }) \mo{2^r} \\
       & \equiv (1 - t^2)^{2^{r-2}}  (1 + 2^{r-1} t(t+1)) \mo{2^r} \, . 
\end{align*}
On the other hand 
\[ P_2(t) = t+1  \]
and 
\[ P_4(t) \equiv (1+2t)(1 - t^2) \mo{4} \, .   \]
\end{thm}

Note that the two expressions for~$P_{2^r}(t)$ are obviously congruent mod $2^r$. In section~\ref{sec-p2} we shall prove the first, while we shall explore the consequences of the second now.

\begin{coro}
Let~$r= v_2(n)$. If~$r \ge 3$ then 
\[ P_n(t) \equiv (1 - t^2)^{\frac{n} {4}} (1 + \frac{n} {2} t (t+1))  \mo{2^r}\, .  \]
For~$r= 1$ one has 
\[ P_n(t) \equiv (1+t)^{\frac{n} {2}} \mo{2} \, ,   \]
Finally for~$r=2$ one has 
\[ P_n(t) \equiv (1 + \frac{n} {2} t)(1 - t^2)^{\frac{n} {4}} \mo{4} \, .   \]
As a result, the congruences announced in theorem~\ref{thm-main} hold. 
%% As a result 
%% %
%% \[ c(n,k) \equiv \left\{ \begin{array}{ll}
%% (-1)^\ell {\frac{n} {4} \choose \ell} & [4] \qquad \textnormal{if}~ n-k = 2 \ell \, , \\
%% (-1)^\ell \frac{n} {2} {\frac{n} {4} \choose \ell} & [4] \qquad \textnormal{if}~ n-k = 2 \ell + 1\, .
%% \end{array}\right.  \]
%% %
\end{coro}

\begin{proof}
Write~$n= 2^r m$. In each case, it is a simple matter of writing 
\[ P_{2^rm} (t) \equiv P_{2^r}(t)^m   \mo{2^r}  \]
and appealing to theorem~\ref{thm-poly-p2}. One uses several times that 
\[ (1 + 2^{r-1}a)^m \equiv 1 + 2^{r-1} ma \mo{2^r}   \]
when~$r \ge 2$.
\end{proof}

\section{Odd primes} \label{sec-p-odd}

In this section we consider a fixed odd prime~$p$, and turn to the proof of the identity 
\[ P_{p^r} (t) \equiv (1 - t^{p-1})^{p^{r-1}} \mo{p^r} \tag{$\dagger$}  \]
for~$r \ge 1$, which is the content of theorem~\ref{thm-poly-odd-prime}. First we note the following consequence when we replace~$t$ by~$pt$:

\begin{lem} \label{lem-p-of-pt}
Suppose~$(\dagger)$ holds for a certain value of~$r$. Then 
\[ P_{p^r}(pt) \equiv 1 \mo{p^{r+1}} \, .  \]
\end{lem} 

\begin{proof}
Rewrite~$(\dagger)$ as 
\[ P_{p^r} (t) = (1 - t^{p-1})^{p^{r-1}} + p^r A(t) \, .   \]
Trying~$t=0$ yields~$A(0) = 0$. As a result $p^rA(pt) \equiv 0 ~ [p^{r+1}]$. Thus we need to show that the coefficients of~$ (1 - (pt)^{p-1})^{p^{r-1}}$ are all divisible by~$p^{r+1}$ (except the first), or equivalently that 
\[ v_p\left[ {p^{r-1} \choose \ell} p^{\ell(p-1)}   \right] \ge r+1  \]
for all~$1 \le \ell \le p^{r-1}$.

As recalled in the introduction, we know the~$p$-adic valuation of such binomial coefficients. If~$\ell = p^a m$, then the valuation above is~$r-1 - a + p^am(p-1) \ge r-1 -a + 2p^a$, using that~$p$ is odd. A simple study of~$x \mapsto -x + 2 p^x$ shows that~$-a + 2p^a \ge 2$, so that $r-1 -a + 2p^a \ge r+1$, as was to be shown.
\end{proof}

The next lemma is also instrumental in the proof of $(\dagger)$.

\begin{lem} \label{lem-an-minus-bn}
Let~$r\ge 1$, let~$a, b \in \z/p^r\z$, and let~$\omega$ be an element of order~$p-1$ in the multiplicative group of units $\left(\z/p^r\z\right)^\times$. Then 
\[ a^{p-1} - b^{p-1} = (a-b)(a-\omega b)(a - \omega ^2b) \cdots (a - \omega^{p-2} b) \, .   \]
The same holds if~$a$ and~$b$ are polynomials in~$\z/p^r\z[t]$.
\end{lem}

\begin{rmk} \label{rmk-counter-example}
It is very tempting to believe that, in any commutative ring~$A$, we have 
\[ a^n - b^n = (a-b)(a-\omega b)(a- \omega ^2b) \cdots (a-\omega^{n-1} b)  \]
for any~$a, b \in A$ and~$\omega \in A^\times$ of order~$n$. However, this is not true in general, as seen with the example~$A= \z/12\z$, $a= -1$, $b = \omega = 7$ and~$n=2$.
\end{rmk}

This lemma will be the object of the next section. For now, we use it together with the previous one in order to prove~$(\dagger)$ by induction on~$r$.

In the case~$r= 1$ we are required to show that 
\[ \prod_{k=0}^{p-1}(1 + kt) = 1 - t^{p-1} \, ,  \]
an equality in~$\f_p[t]$. However, note that both sides are polynomials of the same degree~$p-1$, with constant term~$1$, and both having the same distinct~$p-1$ roots, namely all non-zero elements of~$\f_p$. Indeed if~$x \in \f_p$ is non-zero, then the left hand side vanishes for~$t = x$ since it has a factor~$1 + kx$ with~$k = -1/x$ ; the right hand side vanishes at~$x$, classically, since~$\f_p^{\times}$ has order~$p-1$. Thus these two polynomials, defined over a field, must agree.

Now suppose that~$(\dagger)$ holds for~$r$, and let us prove it for~$r+1$. We start by breaking the defining product for~$P_{p^{r+1}}$ into two factors:

\begin{align*}
P_{p^{r+1}}  (t) = \prod_{k=0}^{p^{r+1} - 1} (1 + kt) & = \left( \prod_{ p | k}(1+kt)  \right) \left( \prod_{p \wedge k = 1} (1+kt) \right) \\
& = P_{p^r} (pt)  \left( \prod_{p \wedge k = 1} (1+kt) \right) \\
& \equiv \prod_{p \wedge k = 1} (1+kt)  \mo{p^{r+1}}
\end{align*}
using lemma~\ref{lem-p-of-pt}. 

It will be convenient to use a group-theoretic notation. Let~$G = \left( \z / p^{r+1} \z \right)^\times$, a multiplicative group of order~$p^r(p-1)$; we are trying to evaluate 
\[ \prod_{k \in G} (1+kt)  \in \z/p^{r+1} \z [t]\, .   \]
Also, let~$H = \langle \omega \rangle$, the group generated by an element~$\omega $ of order~$p-1$. The existence of~$\omega $ is guaranteed by the classical fact that~$G$ is cyclic. For the same reason, $H$ is the unique subgroup of~$G$ having order~$p-1$, and it can be described as the subgroup of~$p^r$-th powers of elements of~$G$, or as the subgroup of elements of order dividing~$p-1$.

Lemma~\ref{lem-an-minus-bn}, used with~$a=1$ and~$b= -kt$, can be stated thus: 
\[ \prod_{\ell \in kH} (1 + \ell t) = 1 - k^{p-1} t^{p-1} \, ,   \]
recalling that~$p$ is odd. Furthermore, we note that the two sides of this equality remain unchanged if~$k$ is replaced by~$kh \in kH$: this is obvious for the left hand side, and for the right hand side one has~$(kh)^{p-1} = k^{p-1}$ since~$h \in H$. We will thus employ a slight abuse of notation and write 
\[ \prod_{k \in G} (1 + kt) = \prod_{k \in G/H} (1 - k^{p-1} t^{p-1}) \, .  \]

Now, an element of~$\left(\z/p^{r+1} \z\right)^\times$ is of the form~$k^{p-1}$ if and only if its order is a power of~$p$, if and only if it maps to~$1$ in~$\left(\z / p \z\right)^\times$ under the natural projection map (which is surjective), if and only if it is of the form~$1 + pi$. More precisely the~$p^r$ different numbers~$k^{p-1}$ as~$k$ ranges over~$G/H$ are exactly the numbers~$1 + pi$ for~$0 \le i < p^r$. In the end 
\[ \prod_{k \in G} (1 + kt) = \prod_{i=0}^{p^r-1} (1 -(1 +pi) t^{p-1}) = \prod_{i=0}^{p^r-1} (1 -t^{p-1}- pit^{p-1}) \, .   \]

The conclusion will come from lemma~\ref{lem-p-of-pt} again, in a slightly different guise. Define 
\[ Q_n(X, Y) = \prod_{i=0}^{n-1}(X + iY) = X^n P_n\left(\frac{Y} {X} \right) \in \z[X, Y] \, .  \]
By lemma~\ref{lem-p-of-pt}, we see that~$Q_{p^r}(X, pY) \equiv X^{p^r} ~ [p^{r+1}]$. Now apply this with~$X= 1 - t^{p-1}$ and~$Y = -t^{p-1}$ and deduce that 
\[ \prod_{k \in G} (1 + kt) \equiv (1 - t^{p-1})^{p^r} \mo{p^{r+1}} \, ,   \]
as required.

We have now established theorem~\ref{thm-poly-odd-prime}, except for lemma~\ref{lem-an-minus-bn} which still requires proof. The next section deals with this.

\section{A ring identity} \label{sec-cyclo}

Let~$A$ be a commutative ring, $a, b \in A$ and~$\omega \in A^\times$ whose order is~$n$. We seek sufficient conditions for the following identity to hold: 
\[ a^n - b^n = (a-b)(a - \omega b)(a - \omega^2 b) \cdots (a - \omega^{n-1}b)\, . \tag{*}  \]
In remark~\ref{rmk-counter-example} we have shown that (*) is not always true. However the conditions we shall give now will ensure that lemma~\ref{lem-an-minus-bn} holds. 

The most obvious remark is probably:

\begin{lem}
If~$A$ is an integral domain, (*) holds.
\end{lem}  

\begin{proof}
In this case~$A$ embeds in its field of fractions, say~$K$. The polynomial 
\[ P = X^n - b^n \in K[X]  \]
cannot have more than~$n$ roots, and we already have~$n$ obvious ones, namely~$\omega^i b$ for~$0 \le i < n$ (this assumes that~$b \ne 0$, the case~$b=0$ being trivial). Thus one can factor~$P$, and recovers (*) upon evaluation at~$X = a$.
\end{proof}

However in our applications we are dealing with rings such as~$A = \z/p^r\z$, so we need to work more. 

Attempting to find a general identity which one could ``evaluate'' to get a formula in any ring~$A$, we end up with the proof of the following criterion.

\begin{prop} \label{prop-crit-cyclo}
Assume that~$\Phi_n(\omega ) = 0$ where~$\Phi_n$ is the~$n$-th cyclotomic polynomial. Then (*) holds.
\end{prop}

Recall that the factorization of~$X^n - 1$ over~$\q[X]$ or over~$\z[X]$ is 
\[ X^n -1 = \prod_{d | n} \Phi_d(X) \, ,  \]
where the cyclotomic polynomials~$\Phi_d$ is the minimal polynomial of a primitive~$d$-th root of unity. In particular~$\Phi_d$ is irreducible over~$\q$. 

\begin{proof}
Put~$F= \q[\Omega ]/ \Phi_n(\Omega )$, which is a field as~$\Phi_n$ is irreducible. We write~$\bar \Omega $ for the image of~$\Omega $ in~$F$, and we note that~$\Phi_n(\bar \Omega ) = 0$ implies~$\bar \Omega ^n = 1$. On the other hand~$\bar \Omega ^d \ne 1$ if~$d$ is a proper divisor of~$n$, for this would imply that some other cyclotomic polynomial~$\Phi_s$ with~$s < n$ vanishes at~$\bar \Omega $ ; however~$\Phi_s$ and~$\Phi_n$ are coprime, so this would lead to a contradiction. Hence the order of~$\bar \Omega $ in~$F^\times$ is just~$n$.

The ring~$F[A, B]$ is an integral domain, where~$A$ and~$B$ are new variables, so in this ring we may apply the lemma and obtain the formula (*) with capitalized~$A, B$ and~$\bar \Omega $.

Lifting this to an identity in~$\q[\Omega, A, B]$ we get 
\[ A^n - B^n = (A - B)(A - \Omega B) \cdots (A - \Omega^{n-1} B) + P \Phi_n\, , \tag{**}  \]
with~$P \in \q[A, B, \Omega ]$. Put differently, the difference~$(A^n - B^n) - (A - B) \cdots (A - \Omega^{n-1} B) $ is divisible by~$\Phi_n(\Omega )$ in~$\q[A, B][\Omega ]$, the quotient being~$P$. At this point, a long division argument shows that in fact~$P \in \z[A, B, \Omega ]$ (using that~$\Phi_n$ is unitary).

Now the identity (**) is entirely in~$\z[A, B, \Omega ]$ and may be evaluated in any ring. Thus if we evaluate at~$A= a$, $B=b$, and~$\Omega = \omega $, where we have arranged to have~$\Phi_n(\omega ) = 0$, we obtain (*).
\end{proof}

We add a little criterion that will guarantee that~$\Phi_n(\omega )= 0$ automatically.

\begin{lem}
Let~$\omega \in A^\times$ have order~$n$. Assume that the image of~$\omega $ in~$(A/M)^\times$ also has order~$n$, for any maximal ideal~$M$ of~$A$. Then~$\Phi_n(\omega ) = 0$.
\end{lem}

\begin{proof}
Start with 
\[ \omega^n - 1 = 0 = \prod_{d|n} \Phi_d(\omega ) \, .   \]
It suffices to show that for~$d \ne n$ the element~$\Phi_d(\omega )$ is invertible, for then this equality can be simplified and we get the result. 

Suppose for a contradiction, then, that there is a~$d \ne n$ such that $\Phi_d(\omega )$ is not invertible, and thus belongs to a maximal ideal~$M$. In the field~$K = A / M$ we certainly have 
\[ \bar \omega ^d - 1 = \prod_{s|d} \Phi_s(\bar \omega ) = 0 \, ,   \]
since~$\Phi_d(\bar \omega ) = 0$, where~$\bar \omega $ is the image of~$\omega $ in~$K$. As a result, the order of~$\bar \omega $ in~$K$ divides~$d$, a contradiction.
\end{proof}

\begin{coro}
Let~$\omega \in \left( \z/p^r\z \right)^\times$ have order~$p-1$, where~$p$ is prime. Then~$\Phi_{p-1}(\omega ) = 0$.
\end{coro}

\begin{proof}
There is just one maximal ideal~$M$ in~$\z/ p^r \z$, and it is the kernel of the natural map~$\z/p^r \z \to \z/ p \z$. So let~$d$ be the order of~$\bar \omega $ in~$\left(\z/p\z\right)^\times$. Then~$d$ is a divisor of~$p-1$, and we need to show that~$d=p-1$ so the last lemma can be applied.

The kernel of the surjective homomorphism 
\[ \left(\z/p^r\z\right)^\times \longrightarrow \left(\z/p\z\right)^\times  \]
has order~$p^{r-1}$. So the order of~$\omega $, which is~$p-1$, must divide~$dp^{r-1}$. However since~$p-1$ is prime to~$p$, we see that~$p-1$ divides~$d$, and we are done.
\end{proof}

We have now all the ingredients to show lemma~\ref{lem-an-minus-bn} from the previous section. Indeed, the last corollary shows that~$\Phi_{p-1}(\omega ) = 0$. Then proposition~\ref{prop-crit-cyclo} can be applied, either to the ring~$A= \z/p^r\z$ or to~$A= \z/ p^r\z[t]$, to deduce that (*) holds for~$n=p-1$. 

\section{The case~$p=2$} \label{sec-p2}

In this section we prove theorem~\ref{thm-poly-p2}, that is, for~$r \ge 3$ we prove that 
\[ P_{2^r} (t) \equiv (1-t^2)^{2^{r-2}} + 2^{r-1} (t+t^2+t^{2^{r-1} +1 } + t^{2^{r-1} +2 }) \mo{2^r} \, .  \tag{$\dagger$}  \]

The arguments are entirely different from those presented in the odd case, and ultimately, simpler. 

Before proving $(\dagger)$, we may at once establish a weaker result:

\begin{lem}
One has 
\[ P_{2^r} (t) \equiv (1-t^2)^{2^{r-2}} \mo{2^{r-1}} \, .    \]
\end{lem}

\begin{proof}
This is immediate by induction on~$r$, using that 
\[ P_{2^r}(t) \equiv P_{2^{r-1}}(t)^2 \mo{2^{r-1}} \, . \qedhere  \]
\end{proof}

We setup some notation. Write 
\[ P_{2^r}(t) = (1 -t^2)^{2^{r-2}} + 2^{r-1} Q_r(t) + 2^{r} R_r(t) \]
where the non-zero coefficients of~$Q_r(t)$ are equal to~$1$. Our goal is to prove that~$Q_r = t+t^2+t^{2^{r-1} +1 } + t^{2^{r-1} +2 }$.

We note the following.

\begin{lem} \label{lem-ind-p2}
For~$r > 1$ one has
\[ P_{2^r}(t) \equiv P_{2^{r-1}}(t) \, P_{2^{r-1}} (-t) \, (1 - 2^{r-1} t) \mo{2^r}\, .   \]
\end{lem}

\begin{proof}
We return to the definition:

\begin{align*}
P_{2^r}(t) & = \prod_{k=0}^{2^r - 1}  \, (1 + kt) \\
          & \equiv \prod_{k= - 2^{r-1}}^{2^{r-1} - 1} \, (1 +kt)    \mo{2^r}
\end{align*}
by shifting~$k$ by~$-2^{r-1}$. The formula follows at once.
\end{proof}

We proceed to prove~$(\dagger)$ by induction on~$r$. For~$r= 3$ one has 
\begin{align*}
P_8(t) & = 5040 t^7 + 13068 t^6 + 13132 t^5 + 6769 t^4 + 1960 t^3 + 322 t^2 + 28 t + 1 \\
      & \equiv 4t^6 + 4t^5 + t^4 + 2t^2 + 4t + 1 \mo{8} \\
      & = (1 - t^2)^{2} + 4(t + t^2 + t^{5} + t^6) \, . 
\end{align*}

Now suppose~$r \ge 4$ and that~$(\dagger)$ has been established for~$r-1$, so that~$Q_{r-1} = t + t^2 + t^{2^{r-2} +1} + t^{2^{r-2} + 2}$. Lemma~\ref{lem-ind-p2} encourages us to estimate $P_{2^{r-1}}(t) \, P_{2^{r-1}} (-t)$, and it turns out that, in our notation,
\[ P_{2^{r-1}}(t) \, P_{2^{r-1}} (-t) \equiv (1-t^2)^{2^{r-2}} + 2^{r-2} (1-t^2)^{2^{r-3}} (Q_{r-1}(t) + Q_{r-1}(-t)) \mo{2^r} \, ,   \]
by a direct computation which takes into account that $2^{r-1}(R_{r-1}(t) + R_{r-1} (-t)) \equiv 0 ~ [2^r]$. The same lemma implies that 
\[ P_{2^r} (t) \equiv (1 - t^2)^{2^{r-2}} + 2^{r-2} (1-t^2)^{2^{r-3}} (Q_{r-1}(t) + Q_{r-1}(-t)) -2^{r-1}t(1-t^2)^{2^{r-2}} \mo{2^r} \]
from which one derives the identity~$(\dagger)$ by a simple rearrangement of the terms.

\section{Application to Chern classes} \label{sec-chern}

As an application of our congruences, we compute some Chern classes. We recall that, if~$G$ is a finite group, one can construct its cohomology ring~$H^*(G, \z)$, which is graded ; when~$\rho \colon G \to GL_N(\C)$ is a representation of~$G$, it has Chern classes~$c_1(\rho )$, $\ldots$, $c_N(\rho )$, with~$c_i(\rho ) \in H^{2i}(G, \z)$. 

In this paper we will restrict ourselves to the case when~$G= C_n$ is a cyclic group of order~$n$, say generated by~$x$. We will let~$\rho $ denote the representation~$G \to \C^\times$ mapping~$x$ to~$e^{\frac{2i \pi} {n}}$. In this situation one has 
\[ H^*(C_n, \z) = \frac{\z[c_1(\rho )]} {n c_1( \rho ) = 0} \, ,   \]
and no further knowledge about cohomology is required from the reader. We will write~$t = c_1(\rho )$, so that computing in~$H^*(G, \z)$ is much like computing in~$\z/n\z[t]$ (except for the constant terms). 

The first obvious computation which we can easily deal with now is that of the Chern classes of the {\em regular representation}. The latter is obtained by letting~$G$ act on itself by multiplication on the left, and then turning this into a representation by taking a complex vector space~$V_n$ with basis in bijection with the elements of~$G$, and extending the action linearly. It is a basic result that 
\[ V_n = \bigoplus_{k=0}^{n-1} \rho ^{k} \, ,   \]
where~$\rho ^k(x) = (\rho (x))^k = e^{\frac{2ki\pi} {n}}$. The usual formulas for Chern classes imply then

\begin{align*}
1 + c_1(V_n) + c_2(V_n) + \cdots + c_n(V_n) = \prod_{k=0}^{n-1}(1 + k c_1(\rho )) = P_n(t) \, . 
\end{align*}
This identity is graded, so that~$c_k(V_n)$ is the degree-$2k$ piece of~$P_n(t)$, that is~$c_k(V_n)= c(n, n-k) t^k$. Thus our results apply, indicating when~$c(n, n-k)$ is~$0$ mod~$n$, thereby indicating when~$c_k(V_n)$ is zero. 

We further specialize to~$n= p^r$ where~$p$ is odd, in order to obtain easy-to-state results. Corollary~\ref{coro-vp-stirling-odd} implies, for example, the following proposition.

\begin{prop}
Let~$V_{p^r}$ denote the regular representation of a cyclic group of order~$p^r$, where~$p$ is an odd prime. Then~$c_k(V_{p^r}) \ne 0$ if and only if~$p-1$ divides~$k$. When~$k= \ell(p-1)$ then, putting~$e= r-2 - v_2(\ell)$, one has~$c_{k}(V_{p^r}) = p^em t^{k}$, where~$m$ is prime to~$p$ and~$t= c_1(\rho )$ as above.
\end{prop}

We can go further and deal with the Chern classes of any {\em permutation representation}. That is, starting with any~$G$-set~$X$, one can also take a complex vector space with a basis in bijection with~$X$ and obtain the representation which we shall write~$\C[X]$.  Next, any~$G$-set~$X$ splits up as the disjoint union of transitive~$G$-sets; correspondingly~$\C[X]$ splits up as a direct sum, so its Chern classes can, classically, be obtained from those of the summands. 

Thus we are reduced to dealing with a~$G$-set of the form~$G/H$. However our restrictions on~$G$ will facilitate the analysis greatly. In fact:

\begin{prop}
When~$H$ is non-trivial, the Chern classes of the representation~$\C[G/H]$ are~$0$ in the cohomology of~$G$.
\end{prop}

\begin{proof}
Since~$G$ is cyclic, $H$ is the only subgroup of its order, call it~$d$, and~$G' = G/H$ is a cyclic group of order~$n' = n/d$, generated by the image~$x'$ of~$x$ (see above). It has a regular representation~$V_{n'}$, and~$\C[G/H]$, which we are trying to understand, is simply~$V_{n'}$ regarded as a representation of~$G$ {\em via} the quotient map~$G \to G'$. 

Above we have seen that the Chern classes of~$V_{n'}$, when viewed in the cohomology of~$G'$, are the monomials of~$P_{n'}(t')$. Here~$t'  = c_1(\rho ')$ where~$\rho '(x') = e^{\frac{2 i \pi} {n'}}$. Writing 
\[ H^*(G', \z) = \frac{\z[t']} {n't' = 0} \, ,   \]
we use the fact that the map 
\[ H^*(G', \z) \longrightarrow H^*(G, \z)  \]
sends~$t'$ to~$dt$. In turn, this is because~$\rho '$, when viewed as a representation of~$G$, is~$\rho^d$, as it sends~$x$ to~$(e^{\frac{2i \pi} {n}})^d$. 

When~$d > 1$, it is a positive power of~$p$, and lemma~\ref{lem-p-of-pt} says that~$P_{p^{r-1}}(pt) \equiv 1 ~ [p^r]$. As a result, the Chern classes of~$V_{n'}$, in the cohomology of~$G$, are~$0$ if~$n' < n$. 
\end{proof}

Together the two propositions proved in this section imply theorem~\ref{thm-chern-intro} from the introduction. Moreover this last proof indicates how, for any cyclic group~$G$ at all, the permutation representations of~$G$ can be analysed {\em via} the regular representations of various quotients~$G'$ of~$G$. Thus the computation of Chern classes for these comes down entirely to a determination of~$c(n,k)$ modulo~$n$ for various values of~$k$ and~$n$, and can be done with the help of our results.

\bibliography{myrefs}
\bibliographystyle{amsalpha}

\end{document}